\documentclass[a4paper,12pt]{amsart}

\usepackage{graphicx}

\newtheorem{theorem}{Theorem}[section]

\newtheorem{lemma}[theorem]{Lemma}
\newtheorem{conjecture}[theorem]{Conjecture}
\newtheorem{metatheorem}[theorem]{Metatheorem}

\newtheorem{Example}[theorem]{Example}

\newtheorem{Remark}[theorem]{Remark}
\newenvironment{remark}{\begin{Remark}\rm}{\end{Remark}}
\newtheorem{Remarks}[theorem]{Remarks}

\newtheorem{Question}[theorem]{Question}

\newtheorem{Summary}[theorem]{Summary}
\newenvironment{summary}{\begin{Summary}\rm}{\end{Summary}}

\newcommand{\al}{\alpha}
\newcommand{\be}{\beta}
\newcommand{\ga}{\gamma}
\newcommand{\Ga}{\Gamma}

\newcommand{\eps}{\varepsilon}

\newcommand{\la}{\lambda}
\newcommand{\La}{\Lambda}

\newcommand{\si}{\sigma}
\newcommand{\Si}{\Sigma}
\newcommand{\ze}{\zeta}

\let\cal=\mathcal
\let\Bbb=\mathbb

\begin{document}

\title[Remainder PPC, Spring 2008]{Lower bound for the
remainder\\ in the Prime-Pair Conjecture}


\date{Spring 2008}

\author{Jacob Korevaar}

\begin{abstract}
Taking $r>0$ let $\pi_{2r}(x)$ denote the number of prime
pairs $(p,\,p+2r)$ with $p\le x$. The prime-pair
conjecture of Hardy and Littlewood (1923) asserts
that $\pi_{2r}(x)\sim 2C_{2r}\,{\rm li}_2(x)$ with an
explicit constant $C_{2r}>0$. A heuristic argument
indicates that the remainder $e_{2r}(x)$ in this
approximation cannot be of lower order than $x^\be$, where
$\be$ is the supremum of the real parts of zeta's zeros.
The argument also suggests an approximation for
$\pi_{2r}(x)$ similar to one of Riemann for $\pi(x)$.
\end{abstract}

\maketitle

\setcounter{equation}{0}    
\section{Introduction} \label{sec:1}
For $r\in{\Bbb N}$ let $\pi_{2r}(x)$ denote the number of
prime pairs $(p,\,p+2r)$ with $p\le x$. The famous
prime-pair conjecture (PPC) of Hardy and
Littlewood \cite{HL23} asserts that for $x\to\infty$,
\begin{equation} \label{eq:1.1}
\pi_{2r}(x)\sim 2C_{2r}{\rm li}_2(x)=
2C_{2r}\int_2^x\frac{dt}{\log^2 t}\sim
2C_{2r}\frac{x}{\log^2 x}.
\end{equation}
Here $C_2$ is the `twin-prime constant',
\begin{equation} \label{eq:1.2}
C_2 = \prod_{p\,{\rm
prime},\,p>2}\,\left\{1-\frac{1}{(p-1)^2}\right\}
\approx 0.6601618,
\end{equation}
and the general `prime-pair constant' $C_{2r}$ is given by
\begin{equation} \label{eq:1.3}
C_{2r} = C_2\prod_{p|r,\,p>2}\frac{p-1}{p-2}.
\end{equation}
No proof of (\ref{eq:1.1}) is in sight, but our
arguments make it plausible that
the best asymptotic estimate for the remainder 
\begin{equation} \label{eq:1.4}
e_{2r}(x)\stackrel{\mathrm{def}}{=}
\pi_{2r}(x)-2C_{2r}{\rm li}_2(x)
\end{equation}
cannot be as small as $x^{1/2}/\log^2 x$;
see Section \ref{sec:10}. 

For the following we set
\begin{equation} \label{eq:1.5}
\be\stackrel{\mathrm{def}}{=} \sup_\rho\,{\rm Re}\,\rho,
\end{equation}
where $\rho$ runs over the complex zeros of
$\ze(s)=\ze(\si+i\tau)$. Recall that Riemann's Hypothesis
(RH) asserts that $\be=1/2$. For the case of the prime
number theorem it is known that the remainder
$$e(x)\stackrel{\mathrm{def}}{=}\pi(x)-{\rm li}(x)
=\sum_{p\le x}\,1-\int_2^x\frac{dt}{\log t}$$
is $\cal{O}(x^{\be+\eps})$ for every $\eps>0$, but
cannot be $\cal{O}(x^{\be-\eps})$ for any $\eps>0$.
Indeed, a formula from Riemann's work suggests
the approximation
\begin{equation} \label{eq:1.6}
\pi(x)={\rm
li}(x) - (1/2)\,{\rm li}(x^{1/2}) - \sum_\rho {\rm
li}(x^\rho) + \cal{O}(x^b)
\end{equation} 
for any $b>\max\{1/3,\be/2\}$. Here the
sum over $\rho$ is a limit of `symmetric' partial sums;
it becomes significant for very large $x$. In 1895 von
Mangoldt obtained the following formula, from which he
derived a proof of (\ref{eq:1.6}); cf.\ Davenport
\cite{Da00}, Edwards \cite{Ed74}:
\begin{equation} \label{eq:1.7}
\psi(x)=\sum_{n\le
x}\,\La(n)=x-\sum_\rho\,\frac{x^\rho}{\rho}
-\frac{\ze'(0)}{\ze(0)}+\sum_k\frac{x^{-2k}}{2k}.
\end{equation}
The formula is exact for all $x>1$ where $\psi(x)$ is
continuous.

For prime pairs $(p,\,p+2r)$ one would
expect that
\begin{align} 
e_{2r}(x) &\ll x^{\be+\eps}\quad\mbox{for
every}\;\;\eps>0,
\quad\mbox{but}\label{eq:1.8} \\
e_{2r}(x) &\ll x^{\be-\eps}\quad\mbox{for no}\;\;\eps>0.
\label{eq:1.9}
\end{align}
Here the symbol $\ll$ is shorthand for the
$\cal{O}$-notation. Some time ago, Dan Goldston
\cite{Go06} suggested that the author's complex method
(now in \cite{Ko07}) might provide a good lower bound
for $e_{2r}(x)$. In this note we use such an approach to
obtain a conditional proof for
\begin{metatheorem} \label{the:1.1}
Statement $(\ref{eq:1.9})$ is correct.
\end{metatheorem} 
For our analysis we introduce an analog to $\psi(x)$:
\begin{equation} \label{eq:1.10}
\psi_{2r}(x)\stackrel{\mathrm{def}}{=}\sum_{n\le
x}\,\La(n)\La(n+2r).
\end{equation}
It is not difficult to see that the PPC (\ref{eq:1.1}) is
equivalent to the asymptotic relation
\begin{equation} \label{eq:1.11}
\psi_{2r}(x)\sim 2C_{2r}x\quad\mbox{as}\;\;x\to\infty.
\end{equation}
For our subsequent analysis it is convenient to work
with the following series of Dirichlet-type, where
$s=\si+i\tau$:
\begin{equation} \label{eq:1.12}
D_{2r}(s)\stackrel{\mathrm{def}}{=}\sum_{n=1}^\infty
\,\frac{\La(n)\La(n+2r)}{n^s(n+2r)^s}
=\int_1^\infty
\frac{d\psi_{2r}(t)}{t^s(t+2r)^s}\qquad(\si>1/2).
\end{equation}
Note that for the boundary behavior of $D_{2r}(s)$ as
$\si\searrow 1/2$, the denominators $n^s(n+2r)^s$ may be
replaced by $n^{2s}$. Hence by a two-way Wiener--Ikehara
theorem for Dirichlet series with positive coefficients,
the PPC in the form (\ref{eq:1.11}) is true if and only
if the difference
\begin{equation} \label{eq:1.13}
G_{2r}(s)=D_{2r}(s)-\frac{2C_{2r}}{2s-1}
\end{equation} 
has `good' boundary behavior as $\si\searrow 1/2$. That
is, $G_{2r}(\si+i\tau)$ should tend to a distribution
$G_{2r}\{(1/2)+i\tau\}$ which is locally equal to a
pseudofunction. By a pseudofunction we mean the
distributional Fourier transform of a bounded function
which tends to zero at infinity; see \cite{Ko05}. It
cannot have poles and is locally given by Fourier series
whose coefficients tend to zero. In
particular $D_{2r}(s)$ itself would have to show
pole-type behavior, with residue $C_{2r}$, for angular
approach of $s$ to $1/2$ from the right; there should be
no other poles on the line $\{\si=1/2\}$.

Heuristic arguments make it plausible that
$D_{2r}(s)$ has a meromorphic extension to some
half-plane $\cal{H}_\eps=\{\si>(\be-\eps)/2\}$ where
$\be=\sup {\rm Re}\,\rho\,$:
\begin{metatheorem} \label{the:1.2}
For every $r\in{\Bbb N}$ there is a number $\eps>0$ such
that 
\begin{equation} \label{eq:1.14}
D_{2r}(s)=\frac{2C_{2r}}{2s-1}
-4C_{2r}\sum_\rho\,\frac{1}{2s-\rho} + H_{2r}(s),
\end{equation}
where $H_{2r}(s)$ is holomorphic in $\cal{H}_\eps$.
\end{metatheorem}
Our approach would take care of Metatheorem \ref{the:1.1}
in the case $\be>1/2$. Metatheorem \ref{the:1.2} suggests
the following approximation for $\psi_{2r}(x)\,$:
\begin{metatheorem} \label{the:1.3}
For each $r$ there is a number $\eta>0$ such that
\begin{equation} \label{eq:1.15}
\psi_{2r}(x)=2C_{2r}x-4C_{2r}\sum_\rho\,x^\rho/\rho+
\cal{O}(x^{\be-\eta}).
\end{equation}
\end{metatheorem}

The case $\be=1/2$ of Metatheorem \ref{the:1.1} is more
subtle. It requires consideration of the function 
\begin{equation} \label{eq:1.16}
\theta_{2r}(x)\stackrel{\mathrm{def}}{=}
\sum_{p,\,p+2r\,{\rm prime};\;p\le x}\,\log^2
p=\int_2^{x+}(\log^2 t)d\pi_{2r}(t),
\end{equation} 
and the associated Dirichlet series
\begin{equation} \label{eq:1.17}
D^0_{2r}(s)\stackrel{\mathrm{def}}{=}\sum_{p,\,p+2r\,{\rm
prime}}\,\frac{\log^2
p}{p^{2s}}=\int_1^\infty\frac{d\theta_{2r}(t)}{t^{2s}}.
\end{equation}
Here our arguments suggest
\begin{metatheorem} \label{the:1.4}
If $\be>1/2$ there is a representation for $D^0_{2r}(s)$
similar to the one for $D_{2r}(s)$. However, if $\be=1/2$
one has 
\begin{equation} \label{eq:1.18}
D^0_{2r}(s)=\frac{2C_{2r}}{2s-1}-\frac{4C^*_{2r}}{4s-1}-
4C_{2r}\sum_\rho\,\frac{1}{2s-\rho}+H^0_{2r}(s),
\end{equation}   
with constants $C^*_{2r}>0$ and a function $H^0_{2r}(s)$
that is holomorphic for $\si>1/4$ and has `good' boundary
behavior as $\si\searrow 1/4$.
\end{metatheorem}
Metatheorems \ref{the:1.2} and \ref{the:1.4} lead to
plausible approximations for $\theta_{2r}(x)$ and
finally, $\pi_{2r}(x)\,$:
\begin{metatheorem} \label{the:1.5}
There are constants $C^*_{2r}>0$ such that
\begin{equation} \label{eq:1.19}
\pi_{2r}(x)= 2C_{2r}{\rm li}_2(x)
-C^*_{2r}{\rm li}_2(x^{1/2})-4C_{2r}\sum_\rho\,\rho\,{\rm
li}_2(x^\rho)+o(x^\be/\log^2 x).
\end{equation}
\end{metatheorem}
The constants $C^*_{2r}$ come from the special case of
the Bateman--Horn conjecture \cite{BH62}, \cite{BH65} that
involves the prime pairs $(p,\,p^2\pm 2r)$: the number
$\pi^*_{2r}(x)$ of such pairs with $p\le x$ should satisfy
an asymptotic relation
\begin{equation} \label{eq:1.20}
\pi^*_{2r}(x)\sim 2C^*_{2r}{\rm li}_2(x)
\quad\mbox{as}\;\;x\to\infty,
\end{equation}
with certain specific constants $C^*_{2r}$. The analysis
in Sections \ref{sec:8}--\ref{sec:10}, which includes
computations by Fokko van de Bult \cite{Bu08},
supports and utilizes 
\begin{metatheorem} \label{the:1.6}
The Bateman--Horn constants $C^*_{2r}$ in
$(\ref{eq:1.20})$ have mean value one {\rm (just like the
Hardy--Littlewood constants $C_{2r}$)}.
\end{metatheorem}

\setcounter{equation}{0}    
\section{Auxiliary functions}
\label{sec:2}
Integration by parts shows that the estimate $e_{2r}(x)
\ll x^{\be-\eps}$ with small $\eps>0$ would be equivalent
to the inequality
\begin{equation} \label{eq:2.1}
e'_{2r}(x)\stackrel{\mathrm{def}}{=}
\theta_{2r}(x)-2C_{2r}x \ll x^{\be-\eps}\log^2 x.
\end{equation}
Note that (\ref{eq:1.17}) and (\ref{eq:2.1}) would imply
holomorphy of the difference
\begin{equation} \label{eq:2.2}
G^0_{2r}(s)=D^0_{2r}(s)-\frac{2C_{2r}}{2s-1}
\quad\mbox{for}\;\;\si={\rm Re}\,s>(\be-\eps)/2.
\end{equation}
Comparison of the series for $D^0_{2r}(s)$ and $D_{2r}(s)$
will show that the difference $D_{2r}(s)-D^0_{2r}(s)$ is
holomorphic for $\si>1/4$; cf.\ Lemma \ref{lem:7.1}
below. Hence an estimate $e_{2r}(x)\ll x^{\be-\eps}$
would imply holomorphy of the difference $G_{2r}(s)$ in
(\ref{eq:1.13}) for $\si>(\be-\eps)/2$, provided
$\be-\eps\ge 1/2$. 

We need precise information on the function $D_0(s)$
derived from (\ref{eq:1.12}).
\begin{lemma} \label{lem:2.1}
For $\si>\frac{1}{2}$ one has
\begin{equation} \label{eq:2.3}
D_0(s) \stackrel{\mathrm{def}}{=}
\sum_{k=1}^\infty\,\frac{\La^2(k)}{k^{2s}}
=\frac{1}{2}\,\frac{d}{ds}
\bigg\{\frac{\ze'(2s)}{\ze(2s)}-\frac{1}{2}\,
\frac{\ze'(4s)}{\ze(4s)}\bigg\}+H_0(s),
\end{equation}
where $H_0(s)$ has an analytic continuation to the
half-plane $\{\si>1/6\}$. This gives a meromorphic
continuation of $D_0(s)$:
\begin{equation} \label{eq:2.4}
D_0(s)=\frac{1}{(2s-1)^2}-\frac{1}{(4s-1)^2}-
\sum_\rho\,\left\{\frac{1}{(2s-\rho)^2}
-\frac{1}{(4s-\rho)^2}\right\}+H_1(s),
\end{equation}
where $H_1(s)$ is holomorphic for $\si>1/6$.
\end{lemma}
\begin{proof} Taking $x={\rm Re}\,z>1$ one has
\begin{equation} \label{eq:2.5}
-\frac{\ze'(z)}{\ze(z)} =
\sum\frac{\La(k)}{k^z} =
\sum_p(\log p)\Big(\frac{1}{p^z}+\frac{1}{p^{2z}}
+\frac{1}{p^{3z}}+\cdots\Big).
\end{equation}
It follows that
\begin{align}
\sum(\log p)p^{-z} &=\sum\La(k)k^{-z}-\sum\La(k)k^{-2z}
+g_1(z)\notag \\ &=
-\ze'(z)/\ze(z)+\ze'(2z)/\ze(2z)+g_1(z), \notag
\end{align}
where $g_1(z)$ is holomorphic for $x>1/3$. Hence
by differentiation,
$$\sum (\log^2 p)p^{-z} = 
\frac{d}{dz}\bigg\{\frac{\ze'(z)}{\ze(z)}
-\frac{\ze'(2z)}{\ze(2z)}\bigg\}-g'_1(z),$$
\begin{align}
\sum_k\La^2(k)k^{-z} &= \sum (\log^2 p)p^{-z}
+ \sum (\log^2 p)p^{-2z}+g_2(z) \notag \\ &=
\frac{d}{dz}
\bigg\{\frac{\ze'(z)}{\ze(z)}-\frac{1}{2}\,
\frac{\ze'(2z)}{\ze(2z)}\bigg\}+g_2(z),\notag
\end{align}
where $g_2(z)$ is also holomorphic for $x>1/3$.
Finally use a standard formula for $(\ze'/\ze)(\cdot)$:
\begin{equation} \label{eq:2.6}
\frac{\ze'(z)}{\ze(z)}=b-\frac{1}{z-1}-\frac{1}{2}\,
\frac{\Ga'(1+z/2)}{\Ga(1+z/2)}+\sum_{\rho}\Big(\frac{1}
{z-\rho}+\frac{1}{\rho}\Big),
\end{equation}
cf.\ Titchmarsh \cite{Ti86}, and set $z=2s$.
\end{proof}
We need the representation in Theorem \ref{the:3.1} below.
It involves sufficiently smooth even
sieving functions $E^\la(\nu)=E(\nu/\la)$ depending on a
parameter $\la>0$. The basic functions
$E(\nu)$ have $E(0)=1$ and support $[-1,1]$; we require
that $E$, $E'$ and $E''$ are absolutely continuous with
$E'''$ of bounded variation. An example involving the
Jackson kernel for ${\Bbb R}$ is given by
\begin{align} 
E^\la(\nu) &= E^\la_J(\nu)=
\frac{3}{4\pi}\int_0^\infty\frac{\sin^4(\la t/4)}
{\la^3(t/4)^4}\cos \nu t\,dt
\notag \\ &=\left\{\begin{array}{ll}
1-6(\nu/\la)^2+6(|\nu|/\la)^3 & \mbox{for
$|\nu|\le\la/2$},\\ 2(1-|\nu|/\la)^3 &
\mbox{for
$\la/2\le|\nu|\le\la$},\\ 0 & \mbox{for
$|\nu|\ge\la$.}
\end{array}\right.\notag
\end{align}

An important role is played by a Mellin transform
associated with the Fourier transform 
$\hat E^\la(t)$. For $0<x={\rm Re}\,z<1$
\begin{align} \label{eq:2.7}
M^\la(z) & \stackrel{\mathrm{def}}{=}
\frac{1}{\pi}\int_0^\infty \hat E^\la(t)t^{-z}dt
= \frac{2}{\pi}\int_0^\infty t^{-z}dt\int_0^\la
E^\la(\nu)(\cos t\nu)d\nu
\notag \\
&= \frac{2}{\pi}\int_0^\la E(\nu/\la)d\nu
\int_0^{\infty-}(\cos\nu t)t^{-z}dt\notag\\ &=
\frac{2}{\pi}\Ga(1-z)\sin(\pi z/2)\int_0^\la
E(\nu/\la)\nu^{z-1}d\nu \\ &=
\frac{2\la^z}{\pi}\Ga(1-z)\sin(\pi z/2)\int_0^1
E(\nu)\nu^{z-1}d\nu \notag \\ &=
\frac{2\la^z}{\pi}\Ga(-z-3)\sin(\pi z/2)\int_0^{1+}
\nu^{z+3}dE'''(\nu).\notag
\end{align}
In the special case of $E^\la_J(\cdot)$ one finds
\begin{align} 
M^\la_J(z) &= \frac{3}{4\pi}\int_0^\infty\frac{\sin^4(\la
t/4)} {\la^3(t/4)^4}\,t^{-z}dt\notag \\ &=
\frac{24}{\pi}\la^z(1-2^{-z-1})\Ga(-z-3)\sin(\pi z/2).
\notag
\end{align}

The function $M^\la(z)$ extends to a
meromorphic function for $x>-3$ with simple poles at the
points $z=1,\,3,\,\cdots$. The residue of the pole at 
$z=1$ is $-2(\la/\pi)A^E$ with $A^E=\int_0^1 E(\nu)d\nu$,
and $M^\la(0)=1$. Furthermore, the standard order
estimates
\begin{equation} \label{eq:2.8}
\Ga(z)\ll |y|^{x-1/2}e^{-\pi|y|/2},\quad\sin(\pi z/2)\ll
e^{\pi|y|/2}
\end{equation}
for $|x|\le C$ and $|y|\ge 1$ imply the useful
majorization
\begin{equation} \label{eq:2.9}
M^\la(x+iy)\ll \la^x(|y|+1)^{-x-7/2}
\quad\mbox{for}\;\;-3<x\le
C,\;\; |y|\ge 1.
\end{equation}

\setcounter{equation}{0}    
\section{A basic representation} \label{sec:3}
The following result is related to Theorem $3.1$ in
\cite{Ko07}, but more precise. It will be verified in
Section \ref{sec:6}.
\begin{theorem} \label{the:3.1}
For any $\la>0$ and $s=\si+i\tau$ with
$1/2<\si<1$ there is a meromorphic representation
\begin{equation}  \label{eq:3.1}
D_0(s)+2\sum_{0<2r\le\la}\,E(2r/\la)D_{2r}(s) 
=V^\la(s)+\Si^\la(s)+H^\la(s).
\end{equation}
Here $D_{2r}(s)$ is given by $(\ref{eq:1.12})$, also for
$r=0$; the functions $D_{2r}(s)$ are holomorphic for
$\si>1/2$. The function $D_0(s)$ has a purely
quadratic pole at $s=1/2$; see $(\ref{eq:2.4})$.
On the basis of the PPC one expects that for
$r\ge 1$, the function $D_{2r}(s)$ has a first-order pole
at $s=1/2$ with residue $C_{2r}$. The functions
$V^\la(s)$ and $\Si^\la(s)$ are described in
$(\ref{eq:3.2})$--$(\ref{eq:3.4})$ below. The error
term $H^\la(s)$ is holomorphic for $0<\si<1$. 
\end{theorem}
The function $V^\la(s)$ is given by the sum
\begin{align} & \label{eq:3.2}
\Ga^2(1-s)M^\la(2-2s)-2\Ga(1-s)\frac{\ze'(s)}{\ze(s)}\,
M^\la(1-s)\sin(\pi s/2) + W^\la(s),\notag \\ &
\quad\mbox{where}\;\;
W^\la(s)=-2\Ga(1-s)\sum_\rho\,\Ga(\rho-s)M^\la(1+\rho-2s)
\sin(\pi\rho/2).
\end{align}
Here $\rho$ runs over the complex zeros of $\ze(s)$. The
combination $V^\la(s)$ is meromorphic for $0<\si<1$,
with poles at $s=1/2$ and the points $s=\rho/2$;
the  apparent poles at the points $s=\rho$ cancel each
other. The simple poles at $s=1/2$ and $s=\rho/2$ have
residues 
\begin{equation} \label{eq:3.3}
A^E\la,\;\;\mbox{and}\;\;-2A^E\la,\;\;
\mbox{respectively,\;with}\;\; A^E=\int_0^1 E(\nu)d\nu.
\end{equation}

The function $\Si^\la(s)$ is given by the sum
\begin{align} & \label{eq:3.4}
\left\{\frac{\ze'(s)}{\ze(s)}\right\}^2 
+2\,\frac{\ze'(s)}{\ze(s)}\,
\sum_{\rho}\,\Ga(\rho-s)M^\la(\rho-s)
\cos\{\pi(\rho-s)/2\} \notag \\ & \quad
+ \sum_{\rho,\,\rho'}\,\Ga(\rho-s)\Ga(\rho'-s) 
M^\la(\rho+\rho'-2s)\cos\{\pi(\rho-\rho')/2\}.
\end{align}
Here $\rho$ and $\rho'$ independently run over the
complex zeros of $\ze(s)$. It is convenient to denote
the sum of the first two terms by $\Si^\la_1(s)$; for
$0<\si\le 1$ it has poles at $s=1$ and at the points
$\rho$. The double series defines a function which we call
$\Si^\la_2(s)$. Under RH the series is absolutely
convergent for $1/2<\si<3/2$. Indeed, setting
$\rho=(1/2)+i\ga$, $\rho'=(1/2)+i\ga'$ and $s=\si+i\tau$,
the inequalities  (\ref{eq:2.8}), (\ref{eq:2.9}) show
that the terms in the double series are
majorized by
\begin{equation} \label{eq:3.5}
C(\la,\tau)(|\ga|+1)^{-\si}(|\ga'|+1)^{-\si}
(|\ga+\ga'|+1)^{-1+2\si-7/2}.
\end{equation} 
Observing that the number of zeros $\rho=(1/2)\pm i\ga$
with $n<\ga\le n+1$ is $\cal{O}(\log n)$, the convergence
now follows from a discrete analog of Lemma \ref{lem:5.1}
below. 

If $\be=\sup {\rm Re}\,\rho>1/2$ there is absolute
convergence for $\be<\si<2-\be$. For $1/2<\si\le\be$
the double sum may be interpreted as a limit of sums over
the zeros $\rho$, $\rho'$ whose imaginary part has
absolute value less than $R$, as $R\to\infty$ through
suitable values; see \cite{Ko07}. By $(\ref{eq:3.1})$ the
apparent poles of $\Si^\la(s)$ at the points $s=\rho$ with
${\rm Re}\,\rho>1/2$ must cancel each other. {\it
Formally}, there is cancellation also at the other points
$\rho$.

\setcounter{equation}{0}    
\section{Metatheorem \ref{the:1.1} for $\be>1/2$ and
Metatheorem \ref{the:1.2}}. \label{sec:4}  
Taking $1/2<\si<1$, formulas
(\ref{eq:3.1})--(\ref{eq:3.5}) show that 
\begin{align} \label{eq:4.1}
\Si^\la_*(s) & \stackrel{\mathrm{def}}{=}\Si^\la(s)-D_0(s)
=2\sum_{0<2r\le\la}\,E(2r/\la)D_{2r}(s) \notag \\ &
-\frac{A^E\la}{s-1/2}+2A^E\la\sum_\rho\,
\frac{1}{s-\rho/2}+H^\la_*(s),
\end{align}
with a `symmetric' sum over $\rho$ and a remainder
$H^\la_*(s)$ that is holomorphic for $0<\si<1$. Recall
from Section \ref{sec:2} that an inequality $e_{2r}(x)\ll
x^{\be-\eps}$ with $\be-\eps\ge 1/2$
would imply holomorphy of the difference
\begin{equation} \label{eq:4.2}
G_{2r}(s)=D_{2r}(s)-\frac{C_{2r}}{s-1/2}
\end{equation}
for $\si>(\be-\eps)/2$. Hence if such holomorphy leads to
a contradiction, so does (\ref{eq:1.9}). This would prove
Metatheorem \ref{the:1.1} for the case $\be>1/2$.

Suppose now that for all $r\le\la/2$ and some $\eps>0$,
the differences $G_{2r}(s)$ 
are holomorphic in the strip $\cal{S}_\eps$ given by
$(\be-\eps)/2<\si<1$. Then by
(\ref{eq:4.1}), the function $\Si^\la_*(s)$ has a
meromorphic continuation [also called $\Si^\la_*(s)$] to
$\cal{S}_\eps$, with poles at $s=1/2$ and some points
$\rho/2$. The pole at $1/2$ will have residue
\begin{equation} \label{eq:4.3}
R(1/2,\la)=2\sum_{0<2r\le\la}\,E(2r/\la)C_{2r}-
\la\int_0^1 E(\nu)d\nu.
\end{equation} 
At this point we use the fact that the prime-pair
constants
$C_{2r}$ have mean value one. Good estimates were
obtained by  Bombieri--Davenport and Montgomery; these
were later improved by Friedlander and Goldston
\cite{FG95} to
\begin{equation} \label{eq:4.4} 
S_m=\sum_{r=1}^m C_{2r}=m-(1/2)\log m+\cal{O}\{\log^{2/3}
(m+1)\}.
\end{equation}
It follows that $R(1/2,\la)$ is $o(\la)$
as $\la\to\infty$, and even $\cal{O}(\log\la)$. Hence by
(\ref{eq:3.4}) the residue at $s=1/2$ of (the meromorphic
continuation of) the double sum $\Si^\la_2(s)$ also is
$o(\la)$. [By (\ref{eq:2.4}) the pole of $D_0(s)$ at
$s=1/2$ is purely quadratic.] The estimate $o(\la)$ is
not surprising if one observes that $\la$ occurs in the
terms of $\Si^\la_2(s)$ only as a factor
$\la^{\rho+\rho'-2s}$; cf.\ (\ref{eq:2.5}). For $\si>1/2$
the exponents have real part $\le 2\be-1$, which is less
than $1$ if $\be<1$.

If the latter kind of heuristic has general validity, the
residues $R(\rho/2,\la)$ of the poles of $\Si^\la_*(s)$ or
$\Si^\la_2(s)$ at the points $\rho/2$ in $\cal{S}_\eps$
must also be $o(\la)$ (at least when $\be<1$ and $\eps$
is small). In view of (\ref{eq:4.1}) this would imply that
many of the functions $D_{2r}(s)$ must become singular at
points $s=\rho/2$ in $\cal{S}_\eps$, which would
contradict our assumption on the differences $G_{2r}(s)$. 

What would be a reasonable hypothesis on the form of the
singularities? Let us start with $0<\la\le 4$ and suppose
that $G_2(s)=D_2(s)-C_2/(s-1/2)$ is holomorphic in
$\cal{S}_\eps$. The residue $R(1/2,\la)$ will equal
$2E(2/\la)C_2-A^E\la$. Thus it changes character as $\la$
passes through the value $2$: it will be linear in $\la$,
of the form $-A^E\la$, for $\la\le 2$, and this linear
term is augmented by the nonlinear term $2E(2/\la)C_2$ as
$\la$ enters the interval $(2,4]$. It is plausible that
the poles of $\Si^\la_*(s)$ at the points $\rho/2$ in
$\cal{S}_\eps$ will be affected in a corresponding
manner. More precisely, the residues $R(\rho/2,\la)$
should change from the linear form $2A^E\la$ to
$2A^E\la-4E(2/\la)C_2$ as $\la$ enters the interval
$(2,4]$. If that is correct, the function
$D_2(s)$ must have first-order poles at the points
$\rho/2$ in $\cal{S}_\eps$ with residue $-2C_2$. The
combination
\begin{equation} \label{eq:4.5}
D_2(s)-\frac{C_2}{s-1/2}+2C_2\sum_{\rho\in S_\eps}\,
\frac{1}{s-\rho/2}
\end{equation}
would be holomorphic in $\cal{S}_\eps$. 

Next taking $4<\la\le 6$ (and if desired, using a
modified function $E(\nu)$ which vanishes on
$[-1/2,1/2]$, say), one may pass to the case
$r=2$, etc. Thus one is led to the postulate
that each function $D_{2r}(s)$ has poles at the points
$\rho/2$ in some strip $\cal{S}_\eps$ with residue
$-2C_{2r}$. If this is correct, the residue of
$\Si^\la_*(s)$ or $\Si^\la_2(s)$ at the poles $\rho/2$ in
$\cal{S}_\eps$ will be  
\begin{equation} \label{eq:4.6}
R(\rho/2,\la)=-4\sum_{0<2r\le\la}\,E(2r/\la)C_{2r}+
2\la\int_0^1 E(\nu)d\nu.
\end{equation} 
Since the constants $C_{2r}$ have average $1$ this would
be consistent with the earlier argument that
$R(\rho/2,\la)$ should be $o(\la)$. 

It follows that Metatheorem \ref{the:1.2} is altogether
plausible, and this suggests Metatheorem \ref{the:1.3}.

\setcounter{equation}{0} 
\section{Integral representations}\label{sec:5}
Setting $z=x+iy$ (and later $w=u+iv$),
we write $L(c)$ for the `vertical line' $\{x=c\}$; the
factor $1/(2\pi i)$ in complex integrals will be omitted.
Thus
$$\int_{L(c)}f(z)dz\stackrel{\mathrm{def}}{=}
\frac{1}{2\pi i}\int_{c-i\infty}^{c+i\infty}f(z)dz.$$
Since it is important for us to have absolutely convergent
integrals, we often have to replace a line $L(c)$ by a
path $L(c,B)=L(c_1,c_2,B)$ with suitable $c_1<c_2$ and
$B>0$:
\begin{equation} \label{eq:5.1}
L(c,B)=\left\{\begin{array}{lllll}
\mbox{$\quad$the half-line} & \mbox{$\{x=c_1,\,-\infty<y\le
-B\}$}\\
\mbox{$+\;$the segment} & \mbox{$\{c_1\le x\le
c_2,\,y=-B\}$}\\
\mbox{$+\;$the segment} & \mbox{$\{x=c_2,\,-B\le y\le B\}$}\\
\mbox{$+\;$the segment} & \mbox{$\{c_2\ge x\ge
c_1,\,y=B\}$}\\
\mbox{$+\;$the half-line} & \mbox{$\{x=c_1,\,B\le
y<\infty\}$;}  
\end{array}\right.
\end{equation} 
cf.\ Figure \ref{fig:1}. 
\begin{figure}[htb]
$$\includegraphics[width=4cm]{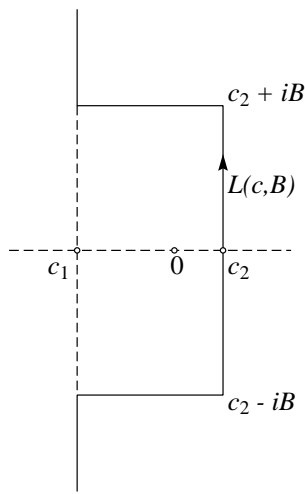}
$$
\caption{The path $L(c_1,c_2,B)$ \label{fig:1}}
\end{figure}
Thus, for example, 
$$\cos\al = \int_{L(c,B)}\Ga(z)\al^{-z}\cos(\pi
z/2)dz \qquad(\al>0),$$
with absolute convergence if $c_1<-1/2$ and $c_2>0$.
Similarly for $\sin\al$. For the combination
$$\cos(\al-\be)t=\cos\al t\cos\be t+\sin\al
t\sin\be t$$  
with $\al,\,\be,\,t>0$, one can now write down an
absolutely convergent repeated integral. 
In \cite{Ko07} it was combined with
(\ref{eq:2.7}) to obtain a repeated complex integral for
the sieving function $E^\la(\al-\be)$ in which $\al>0$
and $\be>0$ occur separately. Taking
$-3<c_1+c'_1<0$, $c_2,\,c'_2>0$, $c_2+c'_2<1$ it was
found that
\begin{align} \label{eq:5.2}
E^\la(\al-\be) & = \int_{L(c,B)}\Ga(z)\al^{-z}dz
\int_{L(c',B)}\Ga(w)\be^{-w}\,\cdot\notag
\\ & \quad\cdot M^\la(z+w)\cos\{\pi(z-w)/2\}\,dw.
\end{align}
We then considered the following integral: 
\begin{align}  \label{eq:5.3}
\quad T^\la(s) &=
\int_{L(c,B)}\Ga(z)\frac{\ze'(z+s)}{\ze(z+s)}\,dz
\int_{L(c,B)}\Ga(w)\frac{\ze'(w+s)}{\ze(w+s)}\,
\cdot \notag \\ & \qquad\cdot 
M^\la(z+w)\cos\{\pi(z-w)/2\}dw,
\end{align}
with suitable paths of integration and for appropriate
$s$; cf.\ Section \ref{sec:6}. Next,
substituting the Dirichlet series for $(\ze'/\ze)(Z)$,
formula (\ref{eq:5.2}) led to the expansion
\begin{align} \label{eq:5.4}
T^\la(s) &=
\sum_{k,l}\La(k)\La(l)k^{-s}l^{-s}E^\la(k-l) \notag
\\ &= D_0(s)+2\sum_{0<d\le\la}\,
\sum_k\,\La(k)\La(k+d)k^{-s}(k+d)^{-s}E^\la(d) \\
&= D_0(s)+2\sum_{0<2r\le\la}\,E(2r/\la)D_{2r}(s)
+H^\la_2(s),\notag
\end{align}
where $H^\la_2(s)$ is holomorphic for $\si>0$. Indeed,
for odd numbers $d$, the product $\La(k)\La(k+d)$ can
be $\ne 0$ only if either $k$ or $k+d$ is of the form
$2^\al$ for some $\al>0$. Thus
$T^\la(s)$ was extended to a holomorphic function on the
half-plane $\{\si>1/2\}$. 

To verify the absolute convergence of the repeated
integral in (\ref{eq:5.2}) we substituted $z=x+iy$,
$w=u+iv$, and used the inequalities (\ref{eq:2.8}),
(\ref{eq:2.9}) together with a simple lemma:
\begin{lemma} \label{lem:5.1}
For real constants $a,\,b,\,c$, the function
$$\phi(y,v)=(|y|+1)^{-a}(|v|+1)^{-b}(|y+v|+1)^{-c}$$
is integrable over ${\Bbb R}^2$ if and only if $a+b>1$,  
$a+c>1$, $b+c>1$ and $a+b+c>2$. 
\end{lemma}
For the convergence of the repeated integral in
(\ref{eq:5.3}) we also used the fact that the quotient
$(\ze'/\ze)(Z)$ grows at most logarithmically in $Y$ for
$X\ge 1$, and for $X\ne 1/2$ under RH; cf.\ (\ref{eq:2.6})
and Titchmarsh \cite{Ti86}. The holomorphy of the integral
for $T^\la(s)$ then followed from locally uniform
convergence in $s$. 
 
The following sections serve as preparations for
the case $\be=1/2$ of Theorem \ref{the:1.1}, so that RH
is satisfied.

\setcounter{equation}{0} 
\section{Derivation of Theorem \ref{the:3.1} under RH}
\label{sec:6}
Changing variables in (\ref{eq:5.3}) one obtains
\begin{align} \label{eq:6.1}
T^\la(s) & = \int_{L(c,B)}\Ga(z-s)
\frac{\ze'(z)}{\ze(z)}\,dz
\int_{L(c,B)}\Ga(w-s)\frac{\ze'(w)}{\ze(w)}\,
\cdot \notag \\ & \quad\; \cdot
\,M^\la(z+w-2s)\cos\{\pi(z-w)/2\}dw,
\end{align}
with new paths $L(c,B)$ and the point $s$ to the left of
them. Using Cauchy's theorem and assuming RH, one may
take $c_1=(1/2)+\eta$, $c_2=1+\eta$ with small $\eta>0$
and $(1/2)+\eta<\si<1+\eta$, $|\tau|<B$. [Without RH one
could take $c_1=1$, $c_2=3/2$ and $1<\si<3/2$.] The
absolute  convergence of the repeated integral follows
from Lemma \ref{lem:5.1}. 

We now move the paths of integration across the poles of
the integrand, the points where $z$ or $w$ is equal to
$1$, $s$ or $\rho$. For the transition one may use
quasi-rectangular contours $W_R$, see Figure
\ref{fig:2}, where $R$ runs through a sequence
$R_n\in(n,n+1)$ such that the horizontal segments at
level $\pm R$ are as far from zeros of the zeta function
as possible. 
\begin{figure}[htb]
$$\includegraphics[width=4cm]{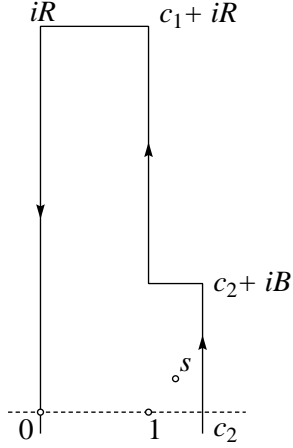}
$$
\caption{Upper half of $W_R$ \label{fig:2}}
\end{figure}
Moving the $w$-path to a line $L(d_1)$ with $d_1\approx
0$, one gets
\begin{equation} \label{eq:6.2}
T^\la(s) =
\int_{L(c,B)}\cdots\,dz\int_{L(d_1)}\cdots\,dw
+U^\la(s)=T^\la_*(s)+U^\la(s),
\end{equation}
say, where by the residue theorem
\begin{equation} \label{eq:6.3}
U^\la(s) = \int_{L(c,B)}
\Ga(z-s)\frac{\ze'(z)}{\ze(z)}\,J(z,s)dz,
\end{equation}
with
\begin{align} \label{eq:6.4}
J(z,s) &=
-\Ga(1-s)M^\la(z+1-2s)\cos\{\pi(z-1)/2\}\notag \\ &
\quad +
\frac{\ze'(s)}{\ze(s)}\,M^\la(z-s)\cos\{\pi(z-s)/2\} \\ &
\quad +\sum_\rho\,\Ga(\rho-s)M^\la(z+\rho-2s)
\cos\{\pi(z-\rho)/2\}.\notag
\end{align}
Observe that for given $s$ with $1/2<\si<1$,
$|\tau|<B$  and small $\eta$, the function $J(z,s)$ is
holomorphic in $z$ on and between the paths $L(c,B)$ and
$L(d_1)$. Defining $J(z,s)$ for $z\in L(c,B)$ by
continuity at  the points $s=1$ and $s=\rho$, it becomes
holomorphic in $s$ for $c_2/2<\si<c_2$. Indeed,
the poles at the point $s=1$ cancel each other, as do the
poles at the points $s=\rho$.

What conditions do $c,\,d$ and $s$ have to
satisfy? The double integral for $T^\la_*(s)$
must be absolutely convergent, which requires
$\si>(c_1+d_1)/2$; cf.\ Lemma \ref{lem:5.1}.
Also, one should not cross a pole of $M^\la(\cdot)$ 
during the shifting operation. Thus $x+u-2\si$ should
remain less than $1$. Taking $\eta$ small, this allows
values of $\si$ close to $1/2$. Since we ultimately want
to consider values of $\si$ around $1/4$, we take $d_1<0$.
Varying $c$ and $d$, the double integral will define
$T^\la_*(s)$ as a holomorphic function for $0<\si<1$
and $|\tau|<B$. 

We next consider the single integral for
$U^\la(s)$. Moving the path $L(c,B)$ across the
points $z=1$, $z=s$ and $z=\rho$ to the line $L(d_1)$, we
obtain the decomposition 
\begin{align} \label{eq:6.5}
U^\la(s) &=
\int_{L(d_1)}\Ga(z-s)\frac{\ze'(z)}{\ze(z)}\,J(z,s)dz 
\notag \\ &
+\bigg\{-\Ga(1-s)J(1,s)+\frac{\ze'(s)}{\ze(s)}\,J(s,s) +
\sum_{\rho'}\,\Ga(\rho'-s)J(\rho',s)\bigg\}.
\end{align}
Working out the residue with the aid of (\ref{eq:6.4})
one obtains nine terms. Five of these combine into the
function $V^\la(s)$ of (\ref{eq:3.2}). Using 
the pole-type behavior of $M^\la(Z)$ at the point
$Z=1$ (Section \ref{sec:2}), the first term in
$V^\la(s)$ provides an important pole at the point
$s=1/2$:
\begin{equation} \label{eq:6.6}
\Ga^2(1-s)M^\la(2-2s)=\frac{A^E\la}{s-1/2}+H^\la_3(s),
\end{equation}
where $H^\la_3(s)$ is holomorphic for $0<\si<1$. The
other terms in $V^\la(s)$ only present simple poles at the
points $s=\rho/2$. A short computation shows that the
residues at those poles are all equal to $-2A^E\la$. The
four remaining terms coming from the big residue
$\{\cdots\}$ provide the function
$\Si^\la(s)$ of (\ref{eq:3.4}). 

It remains to consider the single integral along $L(d_1)$
in (\ref{eq:6.5}), let us call it $U^\la_*(s)$, which we
want to define a holomorphic function in a relatively
wide strip. For that we need absolute convergence of the
`double sum', formed by the $y$-integral along $L(d_1)$
and the sum over $\rho$ in (\ref{eq:6.4}). With
$s=\si+i\tau$ and ${\rm Im\,}\rho=\ga$, the standard
estimates give the following majorant for the integrand:
\begin{align} &
C(\tau)(|y|+1)^{d_1-\si-1/2}\log(|y|+1)\,\cdot 
\notag \\ & \cdot
\sum_\ga\,(|\ga|+1)^{-\si}\la^{d_1+1-2\si}
(|y+\ga|+1)^{-d_1+2\si-4}.\notag
\end{align}
Taking $d_1=-1/2$, the analog of Lemma
\ref{lem:5.1} for the integral of a sum proves the
absolute convergence and holomorphy of the integral when
$0<\si<1$. 

Combination of the above results with (\ref{eq:5.4})
will verify Theorem \ref{the:3.1} under RH.

\setcounter{equation}{0} 
\section{The differences $D_{2r}(s)-D^0_{2r}(s)$ and
$\psi_{2r}(x)-\theta_{2r}(x)$} \label{sec:7}
To treat the case $\be=1/2$ of Theorem
\ref{the:1.1} one has to work with the
function $D^0_{2r}(s)$ of (\ref{eq:1.17}) instead of
$D_{2r}(s)$. In the following $p$ and $q$ are primes that
run over all pairs of the type indicated with the sums. By
the definition of $\La(\cdot)$,
\begin{align} \label{eq:7.1}
D_{2r}(s) &= \sum\frac{\La(n)\La(n+2r)}{n^s(n+2r)^s}
\notag \\ &= \sum_{\substack{p,\,q\\ q-p=2r}}
\frac{\log p\log q}{p^sq^s} +
\sum_{\substack{p,\,q\\ q^2-p=\pm 2r}}
\frac{\log p\log q}{p^sq^{2s}} + H_{1,r}(s),
\end{align}
where $H_{1,r}(s)$ is holomorphic for $\si>1/6$.
The final sum comes from the cases $n=p$, $n+2r=q^2$ and
$n=q^2$, $n+2r=p$. There are only finitely many $n$ of
the form $p^2$ such that $n+2r=q^2$. The function
$g_{1,r}(s)$ includes these and the cases where $n$ or
$n+2r$ is a prime power with exponent $\ge 3$. Continuing
one obtains a sum over the prime pairs $(p,\,p+2r)$ and a
sum over prime pairs $(q,\,q^2\pm 2r)$: 
\begin{lemma} \label{lem:7.1}
One has
\begin{align} \label{eq:7.2}
D_{2r}(s) &= \sum_{p;\,p+2r\,{\rm
prime}}\frac{\log^2 p}{p^{2s}} + 
2\sum_{q;\,q^2\pm 2r\,{\rm prime}}\frac{\log^2
q}{q^{4s}}+H_{2,r}(s) \notag \\ &= D^0_{2r}(s) + 2
D^*_{2r}(s)+H_{2,r}(s),\quad\mbox{say},
\end{align}
where $D^*_{2r}(s)$ and $H_{2,r}(s)$ are holomorphic
for $\si>1/4$, and $\si>1/6$, respectively.
\end{lemma}
We also consider corresponding partial sums,
$\theta_{2r}(x)$ from (\ref{eq:1.16}) and
\begin{equation} \label{eq:7.3}
\theta^*_{2r}(x)\stackrel{\mathrm{def}}{=}\sum_{q\le
x;\,q^2\pm 2r\,{\rm prime}}\,\log^2 q.
\end{equation}
A sieving argument would show that
$\theta^*_{2r}(x)=\cal{O}(x)$; cf.\ \cite{BH65},
\cite{HR74}, \cite{HiRi05}. 
\begin{lemma} \label{lem:7.2}
By $(\ref{eq:1.10})$ and $(\ref{eq:7.1})-(\ref{eq:7.3})$,
\begin{align} \label{eq:7.4}
\psi_{2r}(x) &= \sum_{n\le x}\,\La(n)\La(n+2r) \notag \\ 
&= \theta_{2r}(x) + 2\theta^*_{2r}(x^{1/2})
+\cal{O}(x^{(1/3)}\log^2 x).
\end{align}
\end{lemma}
We can now formulate a refinement of Theorem
\ref{the:3.1}. In view of (\ref{eq:7.2}) the discussion in
Section \ref{sec:6} shows the following.
\begin{theorem} \label{the:7.3}
For $\la>0$ and $1/2<\si<1$, one has
\begin{align} \label{eq:7.5}
T^\la(s) &= D_0(s)+2\sum_{0<2r\le\la}\,E(2r/\la)
\{D^0_{2r}(s)+2D^*_{2r}(s)\} + H^\la_4(s) \notag \\
&= \frac{A^E\la}{s-1/2}-2A^E\la\sum_\rho\,
\frac{1}{s-\rho/2}
+\Si^\la(s)+H^\la_5(s),
\end{align}
where the error terms $H^\la_j(s)$ are holomorphic for
$1/6<\si<1$. 
\end{theorem}
We wish to use (\ref{eq:7.5}) for the study of the
prime-pair functions $D^0_{2r}(s)$ when $\be=1/2$,
and for that we need information on the functions
$D^*_{2r}(s)$ near the line $L(1/4)=\{\si=1/4\}$. This
requires  the consideration of prime pairs $(p,\,p^2\pm
2r)$.

\setcounter{equation}{0} 
\section{Prime pairs $(p,\,p^2\pm 2r)$}
\label{sec:8}
Let $f(p)=p^2-2r$ with $r\in{\Bbb Z}\setminus 0$, and
define
\begin{equation} \label{eq:8.1}
\pi_f(x)=\#\{p\le x:\,f(p)\;\mbox{prime}\}.
\end{equation}
Does $\pi_f(x)$ tend to infinity as $x\to \infty$? Not if
$f(n)$ can be factored, nor if $r\equiv 2\;({\rm
mod\;}3)$, for then $p^2-2r$ is  divisible by $3$ when
$p\ne 3$. However, if $f(n)$ is irreducible and for every
prime $p$, there is a positive integer $n$ such that $p$
does not divide $nf(n)$, one would expect that
$\pi_f(x)\to\infty$ as $x\to \infty$. This is a very
special case of what is usually called
Schinzel's conjecture \cite{SiSc58}. More generally, let
$f(n)$ be any polynomial of degree $d$ with
integer coefficients. For irreducible $f(n)$ we set
\begin{equation} \label{eq:8.2}
N_f(p)=\#\{n,\,1\le n\le p:\,nf(n)\equiv 0\;
({\rm mod}\,p)\},
\end{equation}
and define
\begin{equation} \label{eq:8.3}
C(f)=\prod_p\,\left(1-\frac{1}{p}\right)^{-2}
\left(1-\frac{N_f(p)}{p}\right).
\end{equation} 
The product will converge, but $C(f)$ may be zero; if
$f(n)$ can be factored, we define $C(f)=0$. Then a special
case of the general conjecture of Bateman and Horn
\cite{BH62}, \cite{BH65} asserts the following:
\begin{conjecture} \label{con:8.1}
As $x\to \infty$, one has
\begin{equation} \label{eq:8.4}
\pi_f(x)\sim\frac{C(f)}{d}\,{\rm
li}_2(x)=\frac{C(f)}{d}\,\int_2^x\frac{dt}{\log^2 t}.
\end{equation}
\end{conjecture}
Cf.\ Davenport and Schinzel \cite{DS66}, and Hindry and
Rivoal \cite{HiRi05}. In the special case of the
polynomial
\begin{equation} \label{eq:8.5}
f_{2r}(n)=n^2-2r\qquad (r\in{\Bbb Z}\setminus 0),
\end{equation}
one finds that for $p\not|\, 2r$, using the Legendre
symbol,
\begin{equation} \label{eq:8.6}
N_{f_{2r}}(p)-2=\left(\frac{2r}{p}\right)=\chi(p).
\end{equation}
Here $\chi(p)$ generates a real character (different
from the principal character) belonging to a modulus
$m=m_{2r}$. The convergence of the product for
$C(f_{2r})$ thus follows from the known convergence of
series $\sum_p\,\chi(p)/p$. 

Fokko van de Bult \cite{Bu08} has computed 
\begin{equation} \label{eq:8.7}
C(f_{2})\approx 3.38,
\end{equation}
and counted 
$$\pi^*_2(x)=\pi_{f_2}(x)=\#\{p\le
x:\,p^2-2\;\mbox{prime}\}$$ for
$x=10,\,10^2,\,\cdots,\,10^8$. His results are in
excellent agreement with Conjecture \ref{con:8.1}. In the
table the number
$\pi^*_2(x)$ is compared to rounded values 
$$L^*_2(x)\;\;\mbox{of}\;\;1.69\;{\rm
li}_2(x)=1.69\int_2^x\,\frac{dt}{\log^2 t}.$$ 
The table also gives some ratios 
$$\rho(x)=\pi^*_2(x)/L^*_2(x).$$
These seem to converge to $1$ rather quickly!
\begin{table} \label{table:1}
\begin{tabular}{lllll} 
$x$     & $\pi^*_2(x)$ & $L^*_2(x)$  & $\rho(x)$  & \\
        &              &             &         & \\
$10$    &  4           &             &         & \\
$10^2$  & 13           &             &         & \\
$10^3$  & 52           &             &         & \\
$10^4$  & 259          & 274         & 0.945   & \\ 
$10^5$  & 1595         & 1599        & 0.997   & \\
$10^6$  & 10548        & 10560       & 0.999   & \\
$10^7$  & 74914        & 75223       & 0.996   & \\
$10^8$  & 563533       & 563804      & 0.9995  & \\
        &              &             &         &       
\end{tabular}
\caption{Counting prime pairs $(p,\,p^2-2)$}
\end{table} 

\smallskip
We can now discuss the functions 
\begin{equation} \label{eq:8.8}
D^*_{2r}(s)=\int_1^\infty
\frac{d\theta^*_{2r}(t)}{t^{4s}}\qquad(r\in{\Bbb N})
\end{equation} 
of (\ref{eq:7.2}). Assuming that the Bateman--Horn
conjecture is true for the polynomials
$f_{\pm 2r}(n)=n^2\mp 2r$, one obtains
the following asymptotic relation for the functions
$\theta^*_{2r}(x)$ of (\ref{eq:7.3}):
\begin{equation} \label{eq:8.9}
\theta^*_{2r}(x)=\sum_{q\le x;\,q^2\pm
2r\,{\rm prime}}\,\log^2 q\sim
\frac{C(f_{2r})+C(f_{-2r})}{2}\,x.
\end{equation} 
For us it will be convenient to write this relation in the
form
\begin{equation} \label{eq:8.10}
\theta^*_{2r}(x)\sim 2C^*_{2r}x.
\end{equation}
By the two-way Wiener--Ikehara theorem of \cite{Ko05} and
integration by parts, relation (\ref{eq:8.10}) is
equivalent to the statement that the difference
\begin{equation} \label{eq:8.11}
G^*_{2r}(s)=D^*_{2r}(s)-\frac{2C^*_{2r}}{4s-1}
\end{equation}
has good (that is, pseudofunction) boundary behavior as
$\si\searrow 1/4$. In particular $D^*_{2r}(s)$ must have a
first-order pole at $s=1/4$ with residue
$(1/2)C^*_{2r}$, and no other poles on the line
$\{\si=1/4\}$.

Before returning to the proof of Theorem \ref{the:1.1}
we give a supporting argument for Metatheorem
\ref{the:1.6}, which asserts that the constants
$C^*_{2r}$ have mean value one.

\setcounter{equation}{0} 
\section{A function $T^\la_2(s)$. Metatheorem
\ref{the:1.6}} \label{sec:9}
Using paths specified below we will study the function
\begin{align} \label{eq:9.1}
T^\la_2(s) & = \int_{L(c,B)}\Ga(z-s)
\frac{\ze'(2z)}{\ze(2z)}\,dz
\int_{L(c',B)}\Ga(w-s)\frac{\ze'(w)}{\ze(w)}\,
\cdot \notag \\ & \quad\; \cdot
\,M^\la(z+w-2s)\cos\{\pi(z-w)/2\}dw.
\end{align}
Here analogs to (\ref{eq:5.3}), (\ref{eq:5.4}) provide
the following expansion for $\la>0$, cf.\ (\ref{eq:7.2}):
\begin{align} \label{eq:9.2}
T^\la_2(s) &= \sum_{k,\,l}\,\La(k)\La(l)k^{-2s}l^{-s}
E^\la(k^2-l) \notag \\ &= D^*_0(2s)+
2\sum_{0<2r\le \la}\,E(2r/\la)D^*_{2r}(s)+H^\la_6(s),
\end{align}
where $D^*_0(s)=\sum_p\,(\log^2 p)/p^{2s}$
and $H^\la_6(s)$ is holomorphic for $\si>1/5$. Comparison
with $D_0(s)$ in (\ref{eq:2.3}) shows that 
\begin{equation} \label{eq:9.3}
D^*_0(s) =\frac{1}{(2s-1)^2}-
\sum_\rho\,\frac{1}{(2s-\rho)^2}
-\frac{2}{(4s-1)^2}+H_7(s),
\end{equation}
where $H_7(s)$ is holomorphic for 
$\si>\be/4$. Formula (\ref{eq:9.2}) may be used to define
$T^\la_2(s)$ as a holomorphic function for $\si>1/4$. 

In (\ref{eq:9.1}), assuming RH, one may take
$c_1=(1/4)+\eta$, $c_2=(1/2)+\eta$ and $c'_1=(1/2)+\eta$,
$c'_2=1+\eta$ with small $\eta>0$. Varying $\eta$, the
integral thus represents
$T^\la_2(s)$ as a holomorphic function for
$3/8<\si<1$ and $|\tau|<B$. We now move the $w$-path
$L(c',B)$ across the poles at the points $w=1$, $s$ and
$\rho$ to the path $L(d,B)$, where $d_1=-1/2$ and
$d_2=0$. Then the residue theorem gives  
\begin{equation} \label{eq:9.4}
T^\la_2(s) =
\int_{L(c,B)}\cdots\,dz\int_{L(0)}\cdots\,dw
+U^\la_2(s)=T^{\la,*}_2(s)+U^\la_2(s),
\end{equation}
say, where 
\begin{equation} \label{eq:9.5}
U^\la_2(s) = \int_{L(c,B)}
\Ga(z-s)\frac{\ze'(2z)}{\ze(2z)}\,J(z,s)dz,
\end{equation}
with $J(z,s)$ as in (\ref{eq:6.4}). Recall
that the apparent poles of $J(z,s)$ at the points $s=1$
and $s=\rho$ cancel out. 

We next move the $z$-path $L(c,B)$ in the integral for
$U^\la_2(s)$ to $L(d,B)$. Picking
up residues at $z=s$, $1/2$ and the zeros
$\rho'/2$ of $\ze(2z)$, the result is
\begin{align} \label{eq:9.6} 
U^\la_2(s) &=
\int_{L(d,B)}\Ga(z-s)\frac{\ze'(2z)}{\ze(2z)}\,
J(z,s)dz +V^\la_2(s)\notag \\ &= U^{\la,*}_2(s)+
V^\la_2(s),
\end{align}
say, where
\begin{align} \label{eq:9.7}
V^\la_2(s) &= \frac{\ze'(2s)}{\ze(2s)}\,J(s,s)
-(1/2)\Ga\{(1/2)-s\}J(1/2,s) \notag \\ & \qquad
+ \sum_{\rho'}\,(1/2)
\Ga\{(\rho'/2)-s\}J(\rho'/2,s).
\end{align}
The integrals for $T^{\la,*}_2(s)$ and $U^{\la,*}_2(s)$ in
(\ref{eq:9.4}) and (\ref{eq:9.6}) will define holomorphic
functions for $1/4\le\si<1$.

Let $\cal{S}$ denote the strip $\{1/4<\si<1/2\}$. We
have to know the boundary behavior of $T^\la_2(s)$ as
$\si\searrow 1/4$. What sort of poles on the line
$L(1/4)=\{\si=1/4\}$ will result from the three products
in the formula for $V^\la_2(s)\,$?
The first product involves $J(s,s)$, which by
(\ref{eq:6.4}) is holomorphic on $L(1/4)$,
and $(\ze'/\ze)(2s)$, which has poles at the points
$s=\rho'/2$. The resulting poles have principal parts
\begin{equation} \label{eq:9.8}
\frac{(1/2)J(\rho'/2,\rho'/2)}{s-\rho'/2}.
\end{equation}
Turning to the second product, the function
$J(1/2,s)$ is holomorphic on $L(1/4)$, except for a
simple pole at $s=1/4$ due to the pole of $M^\la(Z)$ for
$Z=1$. The other factor is $-(1/2)\Ga\{(1/2)-s\}$, and by
a short calculation, cf.\ (\ref{eq:2.7}), the principal
part of the pole at $s=1/4$ works out to
\begin{equation} \label{eq:9.9}
\frac{(1/2)A^E\la}{s-1/4},\quad\mbox{where}\;\;A^E=\int_0^1
E(\nu)d\nu.
\end{equation} 

In the third product the function
$J(\rho'/2,s)$ is holomorphic on $L(1/4)$. However, the 
factors $(1/2)\Ga\{(\rho'/2)-s\}$ introduce poles
at the points $s=\rho'/2$. The poles in the
product have principal part 
\begin{equation} \label{eq:9.10}
\frac{-(1/2)J(\rho'/2,\rho'/2)}{s-\rho'/2},
\end{equation}
hence they cancel the poles at the points $s=\rho'/2$
in (\ref{eq:9.8}). The third product also generates a
double series $\Si^\la_{2,2}(s)$: 
\begin{align} \label{eq:9.11}
\Si^\la_{2,2}(s) &\stackrel{\mathrm{def}}{=}
\sum_{\rho,\,\rho'}\,(1/2)\Ga(\rho-s)\Ga\{(\rho'/2)-s\}
\,\cdot \notag \\ & \qquad
\cdot M^\la(\rho-2s+\rho'/2)\cos\{\pi(\rho-\rho'/2)\}.
\end{align}
The series is absolutely convergent for $3/8<\si<1/2$.
Its sum will have an analytic continuation to $\cal{S}$,
also denoted $\Si^\la_{2,2}(s)$, but we do not know much
about its behavior near the line $L(1/4)$; see below.

In support of the hypothesis that the poles of
$V^\la_2(s)$ at the points $s=\rho'/2$ cancel out one
may analyze an integral $T^\la_{1,2}(s)$ related to
$T^\la_2(s)$. It is obtained from (\ref{eq:9.1}) by
interchanging the roles of $(\ze'/\ze)(2\,\cdot)$ and
$(\ze'/\ze)(\cdot)$. The new integral is of course equal
to $T^\la_2(s)$. In the analysis the role of
$J(z,s)$ is now taken by
\begin{align} \label{eq:9.12}
J_2(z,s) &=
\frac{\ze'(2s)}{\ze(2s)}\,M^\la(z-s)\cos\{\pi(z-s)/2\} 
\notag \\ & 
-(1/2)\Ga\{(1/2)-s\}M^\la\{z+(1/2)-2s\}\cos\{\pi(z-1/2)/2\}
\\ & 
+(1/2)\sum_\rho\,\Ga\{(\rho/2)-s\}M^\la\{z+(\rho/2)-2s\}
\cos\{\pi(z-\rho/2)/2\}.\notag
\end{align}
Here the apparent poles at the points $s=1/2$
and $s=\rho/2$ cancel out.
\begin{summary} \label{sum:9.1}
Assume RH. Combination of (\ref{eq:9.2}) and the
subsequent results shows that for 
$3/8<\si<1/2$,
\begin{align} \label{eq:9.13}
T^\la_2(s) &=
D^*_0(2s)+2\sum_{0<2r\le\la}\,E(2r/\la)D^*_{2r}(s)
  + H^\la_8(s) \notag \\ &= \frac{(1/2)A^E\la}{s-1/4}
+\Si^\la_{2,2}(s)+H^\la_9(s),
\end{align}
where $H^\la_8(s)$ and $H^\la_9(s)$ are holomorphic for
$1/4\le\si<1/2$.
\end{summary} 
Observe that the (analytic continuation of the) sum
$\Si^\la_{2,2}(s)$ must have a second-order pole at the
point $s=1/4$. Indeed, $D^*_0(2s)$ has a
quadratic pole at $s=1/4$, see (\ref{eq:9.3}),
and by sieving, the functions $D^*_{2r}(s)$ cannot have
a worse singularity at $s=1/4$ than a first-order pole. In
Section \ref{sec:8} it was made plausible that the
functions $D^*_{2r}(s)$ indeed have a first-order pole at
$s=1/4$. What can we say about the mean value of the
residues $(1/2)C^*_{2r}$, or of the numbers
$C^*_{2r}\,$? By (\ref{eq:9.13}) and (\ref{eq:8.10}) the
residue of $\Si^\la_{2,2}(s)$ at $s=1/4$ is equal to
\begin{equation} \label{eq:9.14}
R^*(\la)=\sum_{0<2r\le\la}\,E(2r/\la)C^*_{2r}-
(\la/2)\int_0^1 E(\nu)d\nu. 
\end{equation}
Now it is plausible that this residue is $o(\la)$ as
$\la\to \infty$. Indeed, $\la$ occurs in the terms of
$\Si^\la_{2,2}(s)$ only as a factor
$\la^{\rho-2s+\rho'/2}$; cf.\ the considerations in
Section \ref{sec:4}. Assuming $R^*(\la)=o(\la)$, and
letting $E(\nu)\le 1$ approach the constant function
$1$ on $[0,1]$, it follows from (\ref{eq:9.14}) that 
\begin{equation} \label{eq:9.15}
\sum_{0<r\le\la/2}\,C^*_{2r}\sim\la/2\quad\mbox{as}\;\;\la\to
\infty.
\end{equation}
Thus the numbers $C^*_{2r}$ should have mean value $1$,
as asserted in Metatheorem \ref{the:1.6}. The metatheorem
is supported by numerical evidence: a computation of the
first fifteen constants $C^*_{2r}$ by Fokko van de Bult
\cite{Bu08} gave their average as $0.98$.
\begin{remark} \label{rem:9.2}
Simple adaptation of our heuristics and accompanying
numerical results indicate that relation (\ref{eq:9.15})
and Metatheorem \ref{the:1.6} can be extended to the case
of prime pairs $(p,\,p^k\pm 2r)$ with $k\ge 3$; see
\cite{FK08}.
\end{remark}

\setcounter{equation}{0} 
\section{Metatheorem \ref{the:1.1} for $\be=1/2$ and
Metatheorem \ref{the:1.4}} 
\label{sec:10}
Taking $1/2<\si<1$, Theorem \ref{the:7.3} shows that 
\begin{align} \label{eq:10.1}
\Si^\la_*(s) & \stackrel{\mathrm{def}}{=}\Si^\la(s)-D_0(s)
=2\sum_{0<2r\le\la}\,E(2r/\la)\{
D^0_{2r}(s)+2D^*_{2r}(s)\} \notag
\\ &\qquad\qquad\qquad\qquad
-\frac{A^E\la}{s-1/2}+2A^E\la\sum_\rho\,
\frac{1}{s-\rho/2}+H^\la_*(s),
\end{align}
where $A^E=\int_0^1 E(\nu)d\nu$. The error term
$H^\la_*(s)$ is holomorphic for $1/6<\si<1$. To complete
the proof of Theorem \ref{the:1.1} we have to deal with
the case $\be=1/2$, so that RH holds. Suppose now that
for $2r\le\la$ and $x\to\infty$,
\begin{equation} \label{eq:10.2}
\theta_{2r}(x)-2C_{2r}x \ll x^{1/2}/\log^2 x.
\end{equation}
Then the corresponding functions 
$G^0_{2r}(s)=D^0_{2r}(s)-2C_{2r}/(2s-1)$
of (\ref{eq:2.2}) have continuous boundary
values for $\si\searrow 1/4$; cf.\ (\ref{eq:1.17}).
 
On the basis of Section \ref{sec:8} we may plausibly
assume that the functions
$G^*_{2r}(s)=D^*_{2r}(s)-2C^*_{2r}/(4s-1)$
show `good' (pseudofunction) boundary behavior for
$\si\searrow 1/4$. Hence by (\ref{eq:10.1}), the function
$\Si^\la_*(s)$ would have a `good' extension to the strip
$1/4\le\si<1$, apart from first-order poles at $s=1/2$,
$1/4$ and the points $\rho/2$. `Good' meaning: holomorphy
for $\si>1/4$ and good boundary behavior after
subtraction of the poles. As in Section
\ref{sec:4}, the pole at $s=1/2$ of $\Si^\la_*(s)$, or of
the double sum
$\Si^\la_2(s)$ in (\ref{eq:3.4}), will have residue
$R(1/2,\la)$ as in (\ref{eq:4.3}). By the mean-value
property of the constants $C_{2r}$ this residue is
$o(\la)$ as $\la\to\infty$. We recall that this was not
surprising because $\la$ occurs in the terms of
$\Si^\la_2(s)$ only as a factor $\la^{\rho+\rho'-2s}$. 

Since by our assumption (\ref{eq:10.2}) the functions
$D^0_{2r}(s)$ would have no pole at $s=1/4$, the pole of
$\Si^\la_*(s)$ or $\Si^\la_2(s)$ at that point would have
residue
\begin{equation} \label{eq:10.3}
R(1/4,\la)=2\sum_{0<2r\le\la}\,E(2r/\la)C^*_{2r}=
(\la/2)\int_0^1 E(\nu)d\nu+R^*(\la),
\end{equation} 
with $R^*(\la)$ as in (\ref{eq:9.14}). In Section
\ref{sec:9} it was made plausible that $R^*(\la)=o(\la)$
as $\la\to\infty$. We used both numerical
evidence {\it and} the argument that the terms of the
double sum $\Si^\la_{2,2}(s)$ contain $\la$ only as a
factor $\la^{\rho-2s+\rho'/2}$. However, the latter
argument would {\it also suggest} that
$R(1/4,\la)=o(\la)$. Indeed, the terms in the double
series $\Si^\la_2(s)$ of (\ref{eq:3.4}) contain $\la$
only as a factor $\la^{\rho+\rho'-2s}\,$! 

The contradiction indicates that assumption 
(\ref{eq:10.2}) is false, and that formula (\ref{eq:10.3})
for $R(1/4,\la)$ is incorrect. It is most likely that the
functions $D^0_{2r}(s)$ have poles at the point $s=1/4$,
and that these poles more or less cancel those of the
functions $2D^*_{2r}(s)$. Thus the true residue
$R(1/4,\la)$ of $\Si^\la_2(s)$ at the point
$s=1/4$ may still be $o(\la)$ as $\la\to\infty$. Note
also that by (\ref{eq:10.1}), the (true) residue
$R(1/4,\la)$ is equal to $0$ for $0<\la\le 2$. Combining
our observations, the simplest hypothesis would be 
that $\Si^\la_*(s)$ {\it does not have a pole at} $s=1/4$
for any value of $\la$! Letting $\la$ increase from $2$
on, it would follow that $D^0_{2r}(s)$ has a pole at
$s=1/4$ with residue $-C^*_{2r}$ for every $r$. This
contradiction to (\ref{eq:10.2}) would establish
Metatheorem \ref{the:1.1}! 

One could also argue on the basis of the points
$s=\rho/2$. Since $D^*_{2r}(s)$ would have no poles at
those points, assumption (\ref{eq:10.2}) would require
poles of $\Si^\la_*(s)$ or $\Si^\la_2(s)$ at $s=\rho/2$
with residue $2A^E\la$. But this would contradict the
assumption that the residues are $o(\la)$ which was
reasonable because the terms of $\Si^\la_2(s)$ contain
$\la$ only as a factor $\la^{\rho+\rho'-2s}$. Thus
(\ref{eq:10.2}) must be incorrect for many values of $r$.
The simplest explanation of a residue $o(\la)$ for 
$\Si^\la_2(s)$ would be that the functions $D^0_{2r}(s)$
have poles at $s=\rho/2$ with residue $-2C_{2r}$. Indeed,
we know that 
$$-2\sum_{0<2r\le\la}\,E(2r/\la)2C_{2r}+2A^E\la=o(\la)
\quad\mbox{as}\;\;\la\to \infty.$$ 

We now turn to Metatheorem \ref{the:1.4}. Using Lemma
\ref{lem:7.1}, the preceding arguments make it plausible
that, indeed,
\begin{align} \label{eq:10.4}
D^0_{2r}(s) &= D_{2r}(s)-2D^*_{2r}(s)-H_{2,r}(s) \notag
\\ &=
\frac{2C_{2r}}{2s-1}-\frac{4C^*_{2r}}{4s-1}-
4C_{2r}\sum_\rho\,\frac{1}{2s-\rho}+H^0_{2r}(s),
\end{align}   
where $H^0_{2r}(s)$ is holomorphic for $\si>1/4$ and has
good boundary behavior for $\si\searrow 1/4$.

In the case $\be=1/2$ Metatheorem \ref{the:1.4} suggests
the approximation
\begin{equation} \label{eq:10.5}
\theta_{2r}(x)=2C_{2r}x - 4C^*_{2r}x^{1/2}
- 4C_{2r}\sum_\rho\,x^\rho/\rho+o(x^{1/2}).
\end{equation}
Finally, to arrive at Metatheorem \ref{the:1.5} one would
use the formula
$$\pi_{2r}(x)=\int_2^x \frac{d\theta_{2r}(t)}{\log^2 t}.$$

\bigskip

\noindent{\scshape KdV Institute of Mathematics, 
University of  Amsterdam, \\
Plantage Muidergracht 24, 1018 TV Amsterdam, Netherlands}

\noindent{\it E-mail address}: 
{\tt korevaar@science.uva.nl}


\begin{thebibliography}{[16]}

\bibitem{BH62} P.\ T.\ Bateman and R.\ A.\ Horn, {\it A
heuristic asymptotic formula concerning the distribution
of prime numbers}. Math.\ Comp.\ {\bf 16} (1962),
363--367.  [{\it sec \ref{sec:1}, \ref{sec:8}}]

\bibitem{BH65} P.\ T.\ Bateman and R.\ A.\ Horn,
{\it Primes represented by irreducible polynomials in one
variable}. Proc.\ Sympos.\ Pure Math., vol.\ VIII
 pp 119--132. Amer.\ Math.\ Soc., Providence, R.I., 1965.
[{\it sec \ref{sec:1}, \ref{sec:7}, \ref{sec:8}}]

\bibitem{Bu08} F.\ J.\ van de Bult,
{\it Counting prime pairs $(p,\,p^2\pm 2r)$}. In e-mails
of December 2007 and January 2008. [{\it sec
\ref{sec:1}, \ref{sec:8}, \ref{sec:9}}]

\bibitem{FK08} Fokko van de Bult and Jaap Korevaar, {\it
Mean value one of prime-pair constants}. Manuscript,
Amsterdam, June 2008. See arXiv:0806.1667v1 [math.NT].
[{\it sec \ref{sec:9}}]

\bibitem{Da00} H.\ Davenport, {\it Multiplicative number
theory}. (Third edition, revised  by H.\
L.\ Montgomery.) Graduate Texts in Math., 74. Springer,
New York, 2000. [{\it sec \ref{sec:1}}]

\bibitem{DS66} H.\ Davenport and A.\ Schinzel, {\it A
note on certain arithmetical constants}. Illinois J.\ 
Math.\ {\bf 10} (1966), 181--185. [{\it sec \ref{sec:8}}]

\bibitem{Ed74} H.\ M.\ Edwards, {\it Riemann's zeta
function}. Academic Press, New York, 1974. Reprinted by
Dover Publications, Mineola, N.Y., 2001. [{\it sec
\ref{sec:1}}]

\bibitem{FG95} J.\ B.\ Friedlander and D.\ A.\ Goldston,
{\it Some singular series averages and the distribution of
Goldbach numbers in short intervals}. Illinois J.\ Math.\
{\bf 39} (1995), 158--180. [{\it sec \ref{sec:4}}]

\bibitem{Go06} D.\ Goldston, {\it A suggestion}. In e-mail
of February 2006. [{\it sec \ref{sec:1}}]

\bibitem{HR74} H.\ Halberstam, and H.-E.\ Richert,
{\it Sieve methods}. Academic Press, London, 1974. [{\it
sec \ref{sec:7}}]

\bibitem{HL23} G.\ H.\ Hardy and J.\ E.\ Littlewood,
{\it Some problems of `partitio numerorum'. III: On the
expression  of a number as a sum of primes}. Acta Math.\
{\bf 44} (1923), 1--70. [{\it sec \ref{sec:1}}]

\bibitem{HiRi05} M.\ Hindry and T.\ Rivoal, {\it Le
$\La$-calcul de Golomb et la conjecture de Bateman--Horn}.
Enseign.\  Math.\ (2) {\bf 51} (2005), 265--318. [{\it sec
\ref{sec:7}, \ref{sec:8}}]

\bibitem{Ko05} J.\ Korevaar, {\it Distributional
Wiener--Ikehara theorem and twin primes}. Indag.\ Math.\
(N.S.) {\bf 16} (2005), 37--49. [{\it sec \ref{sec:1},
\ref{sec:8}}]

\bibitem{Ko07} J.\ Korevaar, {\it Prime pairs and zeta's
zeros}. Manuscript, Amsterdam, May 2007. 
See arXiv:0806.0934v1 [math.NT]. [{\it sec \ref{sec:1},
\ref{sec:3}, \ref{sec:5}}]

\bibitem{SiSc58} A.\ Schinzel and W.\ Sierpinski, {\it
Sur certaines hypoth\`{e}ses concernant les nombres
premiers}. Acta Arith.\ {\bf 4} (1958), 185--208. 
[{\it sec \ref{sec:8}}]

\bibitem{Ti86} E.\ C.\ Titchmarsh, {\it The Theory of the
Riemann  Zeta-Function}. First edition 1951, second
edition edited by D.\ R.\ Heath-Brown, Clarendon Press,
Oxford, 1986. [{\it sec \ref{sec:2}, \ref{sec:5}}] 

\end{thebibliography}
\enddocument